\documentclass[preprint,11pt]{elsarticle}

\usepackage{hyperref}
\usepackage{floatrow,booktabs,caption}
\usepackage{amstext,amsfonts,amsmath,amssymb,array,tikz}
\usepackage{amsthm,moresize,mathtools}
\usepackage{graphicx}

\makeatletter
\def\@author#1{\g@addto@macro\elsauthors{\normalsize%
		\def\baselinestretch{1}%
		\upshape\authorsep#1\unskip\textsuperscript{%
			\ifx\@fnmark\@empty\else\unskip\sep\@fnmark\let\sep=,\fi
			\ifx\@corref\@empty\else\unskip\sep\@corref\let\sep=,\fi}%
		\def\authorsep{\unskip,\space}%
		\global\let\@fnmark\@empty
		\global\let\@corref\@empty
		\global\let\sep\@empty}%
	\@eadauthor={#1}
}
\makeatother

\def\dj{d\kern-0.4em\char"16\kern-0.1em}

\DeclareMathOperator{\md}{mod}

\newtheorem{thm}{Theorem}[section]
\newtheorem{lem}[thm]{Lemma}

\newtheorem{cor}[thm]{Corollary}
\newtheorem{conj}[thm]{Conjecture}
\theoremstyle{definition}
\newtheorem{example}[thm]{Example}

\newcommand{\ba}{\begin{array}}
\newcommand{\ea}{\end{array}}

\newcommand{\bc}{\begin{center}}
\newcommand{\ec}{\end{center}}

\journal{the journal.}

\begin{document}
	
\begin{frontmatter}
		
\title{\bf On joins of a clique and a co-clique as star complements in regular graphs}

\author{Yuhong Yang}
\ead{yuhong$\_{}$yangsu@163.com}
\address{College of Mathematics and Systems Science, Xinjiang University, Urumqi, Xinjiang 830046, China}
	
\author{Jianfeng Wang\corref{mycorrespondingauthor}}\cortext[mycorrespondingauthor]{Corresponding author}
\ead{jfwang@sdut.edu.cn}
\address{School of Mathematics and Statistics, Shandong University of Technology, Zibo 255049, China}
	
\author{Qiongxiang Huang}
\ead{huangqx@xju.edu.cn}
\address{College of Mathematics and Systems Science, Xinjiang University, Urumqi, Xinjiang 830046, China}
	
\author{Zoran Stani\'c}
\ead{zstanic@matf.bg.ac.rs}
\address{Faculty of Mathematics, University of Belgrade,Studentski trg 16, 11 000 Belgrade, Serbia}

\begin{abstract}
In this paper we consider $r$-regular graphs $G$ that admit the vertex set partition such that one of the induced subgraphs is the join of an $s$-vertex clique and a $t$-vertex co-clique and represents a star complement for an eigenvalue $\mu$ of $G$. The cases in which one of the parameters $s, t$  is less than 2 or $\mu=r$ are already resolved.
It is conjectured in [J. Wang, X. Yuan, L. Liu, Regular graphs with a prescribed complete multipartite graph as a star complement, Linear Algebra Appl.~579 (2019) 302--319] that if $s, t\geq 2$ and $\mu\neq r$, then $\mu=-2, t=2$ and $G=\overline{(s+1)K_2}$. For $\mu=-t$ we verify this conjecture to be true. We further study the case in which $\mu\neq-t$ and confirm the conjecture provided $t^2-4\mu^2t-4\mu^3=0$. For the remaining possibility we determine the structure of a putative counterexample and relate its existence to the existence of a particular 2-class block design. It occurs that the smallest counterexample would have 1265 vertices.
\end{abstract}


\end{frontmatter}
	
\section{Introduction}\label{se-1}

If $\mu$ is an eigenvalue of (the adjacency matrix $A(G)$ of) a finite simple graph $G$ with multiplicity $k$, then a \textit{star complement} for $\mu$ in $G$ is the induced subgraph $G-X$ such that $|X|=k$ and $\mu$ is not an eigenvalue of $G-X$. In this situation $X$ is called a \textit{star set} for $\mu$ in $G$.  The main properties of star complements can be found in \cite[Chapter~5]{Cvet}.
	
Graphs with prescribed star complements have been extensively studied in \cite{NE,Cvet1,Rame,Rowlinson2,Peter1,Peter2,Rowlinson,Zoran}. In particular, Jackon and Rowlinson~\cite{PS} characterized regular graphs with the complete bipartite graph $K_{2,5}$ as a star  complement, Asgharsharghi and Kiani~\cite{LA} characterized regular graphs with the complete tripartite graph $K_{r,r,r}$ as a star complement, Yuan et al.~\cite{Wang1} determined maximal graphs with $K_{1,t}$ in the role of a star complement for $\mu=-2$ and also described  regular graphs with the complete  bipartite $K_{2,t}$ graph as a star complement for any  eigenvalue. Trees and complete graphs as star complements for $1$ as the second largest eigenvalue are characterized  by Stani\' c \cite{Stan}.
		
We use $K_n$ and $nK_1$ to denote the clique (i.e.~the complete graph) and the co-clique (the graph without edges) of order $n$, respectively. The disjoint union of the graphs $G$ and $H$ is denoted by $G\cup H$. The join $G\nabla H$ is obtained by inserting an edge between every vertex of $G$ and every vertex of $H$. The complement of $G$ is denoted by $\overline{G}$. For the remaining terminology and notation, we refer the reader to \cite{Cvet,Zoran}.
		
The graph $K_s\nabla tK_1~(=\overline{sK_1\cup K_t})$ in the role of a star complement has received a great deal of attention in the recent years. For $t=1$ the star complement reduces to the complete graph $K_{s+1}$, and this case is completely resolved in \cite{Zoran1,Wang}.  Stani\' c~\cite{Zoran1} considered strongly regular graphs with this star complement,  proved that they do not exist for $s,t\geq 2$ and provide some examples for $s=1, t\geq 2$.   Rowlinson and Tayfeh-Rezaie~\cite{Rowlinson} characterized all regular graphs with $K_{1,t}$ as a star complement.  Wang et al.~\cite{Wang}  determined all $r$-regular graphs with $K_s\nabla tK_1$ as a star complement for $r$, and formulated the following conjecture.
		
\begin{conj}[\cite{Wang}]\label{con-1}
If an $r$-regular graph $G$ has $K_s\nabla tK_1$ $(s,t\geq 2)$ as a star complement for an  eigenvalue $\mu\neq r$, then $\mu=-2$, $t=2$ and $G=\overline{(s+1)K_2}$.
\end{conj}
	
In this paper we prove that the conjecture holds for $\mu=-t$. For $\mu\neq-t$ we confirm the conjecture under the additional assumption that $t^2-4\mu^2t-4\mu^3=0$. For $t^2-4\mu^2t-4\mu^3\neq 0$ we prove that $\mu$ must be a positive integer, determine the structure of a putative counterexample and relate its existence to the existence of a particular 2-class block design. We also list the sets of feasible parameters of such a graph for $\mu\leq 800$. It occurs that the smallest counterexample would have 1265 vertices.
	
Section~\ref{sec:prel} is preparatory and mostly related to star complements. Our contribution is reported in Sections~\ref{sec:main} and~\ref{sec:rel}.

\section{Preliminaries}\label{sec:prel}

We fix some notation and recall some results related to star complements in graphs. For a subset $X$ of the vertex set $V(G)$ of a graph $G$, we write $G[X]$ to denote the graph induced by~$X$. The first result is the well-known Reconstruction Theorem.
	\begin{thm}[Theorem 5.1.7, \cite{Cvet}]\label{thm-1}
		Let $X$ be a set of $k$ vertices in a graph $G$ and suppose that~$G$ has adjacency matrix
		\begin{equation}\label{eq:RT}A(G)=
		\begin{pmatrix}
			A_{X} & B^\intercal  \\
			B & C \\
		\end{pmatrix},\end{equation}
	    where $A_{X}$ is the adjacency matrix of $G[X]$. We have:
		\begin{itemize}
			\item[(i)]  $X$ is a star set for $\mu$ in $G$ if and only if $\mu$ is not an eigenvalue of $G-X$ and
			\begin{equation}\label{equation-1}
				\mu I-A_{X}=B^\intercal (\mu I-C)^{-1}B.
			\end{equation}
		
			\item[(ii)] If $X$ is a star set for $\mu$, then the eigenspace of $\mu$ consists of vectors
			$\begin{pmatrix}
				\mathbf{x} \\
				(\mu I-C)^{-1}B\mathbf{x}
			\end{pmatrix},$ \mbox {for $\mathbf{x}\in \mathbb{R}^k$}.
		\end{itemize}
	\end{thm}
	
{With the notation of Theorem~\ref{thm-1}, let $H=G-X$. It is clear that $H$ is a subgraph of~$G$ induced by $\overline{X}=V(G)\backslash X$, with $|\overline{X}|=n-k$ and $A(H)=C$. For $u\in X$, denote by $\mathbf{b}_{u}$  the vector-column of $B$ corresponding to $u$. Obviously, $\mathbf{b}_{u}$ is the characteristic vector of the $H$-neighbourhood $N_{H}(u)$ of $u$.
	
	We define the bilinear form on $\mathbb{R}^{n-k}$ by  $\langle \mathbf{x}, \mathbf{y}\rangle=\mathbf{x}^\intercal (\mu I-A(H))^{-1}\mathbf{y}$. By direct computation we get
	$$\langle \mathbf{b}_{u}, \mathbf{b}_{v}\rangle=\left\{\begin{array}{rl}\mu & \mbox{if $u=v$,}\\
		-1& \mbox{if $u\sim v$,}\\
		0 &  \mbox{otherwise.}
	\end{array}\right.
	$$
	Obviously, for $u\neq v$, $\mathbf{b}_{u}\not=\mathbf{b}_{v}$ provided $\mu\notin\{0, -1\}$, which leads to the following result.
	
	\begin{lem}[\cite{Peter}]\label{lem-5}
		Let $X$ be a star set for $\mu$ in $G$. If $\mu\not=0$, then $N_{H}(u)$ $(u\in X)$ is non-empty and if $\mu\not\in\{-1,0\}$, then  $N_{H}(u)\neq N_{H}(v)$ for distinct $u, v\in X$.
	\end{lem}
	
	We conclude this section by the two more results taken from \cite{Cvet}.
	
	\begin{lem}[Proposition 5.2.1, \cite{Cvet}]\label{lem-55}
		Let $C$ be a square matrix with minimal polynomial
		$$m(x)=x^{d+1}+c_{d}x^{d}+c_{d-1}x^{d-1}+\cdots+c_1x+c_0.$$
		If $\mu$ is not an eigenvalue of $C$, then
		$$m(\mu)(\mu I-C)^{-1}=a_dC^d+a_{d-1}C^{d-1}+\cdots+a_1C+a_0I,$$
		where $a_d=1$ and for $0<i\leq d$,
		$$a_{d-i}=\mu^i+c_d\mu^{i-1}+c_{d-1}\mu^{i-2}+\cdots+c_{d-i+1}.$$
	\end{lem}

		\begin{lem}[Proposition 5.2.4,  \cite{Cvet}]\label{lem-4}
		If $X$ is a star set in an $r$-regular graph $G$ for an eigenvalue $\mu\neq r$, then $\langle\mathbf{b}_{u},\mathbf{j}\rangle=-1$ for all $u\in X$.
	\end{lem}

	\section{$K_s\nabla tK_1$ $(s,t\geq2)$ as a star complement in a regular graph}\label{sec:main}
	In this section  we consider a regular graph $G$ that has a star complement $H=K_s\nabla tK_1$ $(s,t\geq2)$ for an eigenvalue $\mu$. After a suitable labelling $A(G)$ can be expressed as in \eqref{eq:RT} along with
	\begin{equation}\label{X-eq-1}A(H)=\begin{pmatrix}
			J_{s\times s}-I_{s\times s} & J_{s\times t} \\
			J_{t\times s} & O_{t\times t}
		\end{pmatrix}.
	\end{equation}
	Now we use the same notations as in \cite{Wang}. Let $V(H)=V(K_s)\cup V(tK_1)=R_1\cup R_2\cup\cdots\cup R_s\cup V(tK_1)$, where $|R_i|=1$ for  $1\leq i\leq s$. A vertex $u\in X$ is of type $(c_1,c_2,\ldots,c_s,b)$ if it has $c_i\in\{0,1\}$ neighbours in  $R_i$ and $b$ neighbours in $V(tK_1)$. It is clear that $|N_{tK_1}(u)|=b$. If we set $\sum\limits_{i=1}^s c_i=a$, then we also have $|N_{K_s}(u)|=a$. We first quote a lemma which will be used in the sequel.

	\begin{lem}[Theorem~3.3, \cite{Wang}]\label{lem-3-1}
		Suppose that an $r$-regular graph $G$ contains the star complement $K_s\nabla tK_1$ $(s,t\geq2)$ for an eigenvalue $\mu$, with the corresponding star set $X$. If  all vertices in $X$ are of type $( c_1,c_2,\ldots,c_s,b)$ with $\sum\limits_{i=1}^sc_i=s-1$, then $\mu=-2$, $b=t=2$ and $G=\overline{(s+1)K_2}$.
	\end{lem}
	
	We now use Lemma~\ref{lem-55} to compute $m(\mu)(\mu I-A(H))^{-1}$. From (\ref{X-eq-1}) we have
	$$
	A(H)^2=\begin{pmatrix}
		(s+t-2)J_{s\times s}+I_{s\times s} & (s-1)J_{s\times t} \\
		(s-1)J_{t\times s} & sJ_{t\times t}
	\end{pmatrix}
	$$
	and
	$$
	A(H)^3=\begin{pmatrix}
		(s^2+2st-3s-2t+3)J_{s\times s}-I_{s\times s} & (s^2+st-2s+1)J_{s\times t} \\
		(s^2+st-2s+1)J_{t\times s} & (s^2-s)J_{t\times t}
	\end{pmatrix}.
	$$
	It follows that the minimal polynomial of $A(H)$ is given by
	$$m(x)=x(x+1)(x^2-(s-1)x-st)=x^4+(2-s)x^3+(1-s-st)x^2-stx.$$
	Since $\mu$ is not an eigenvalue of $H$, we have $\mu\notin\{0,-1\}$ and $\mu^2-(s-1)\mu-st\not=0$. Then, with the notation of Lemma \ref{lem-55}, we have
	$$
	\left\{\begin{array}{l}
		c_3=2-s, \\
		c_2=(1-s-st), \\
		c_1=-st, \\
		c_0=0,
	\end{array}\right.
	$$
	which leads to
	$$
	\left\{\begin{array}{l}
		a_3=1, \\
		a_2=\mu+s-2, \\
		a_1=\mu^2+(2-s)\mu+1-s-st, \\
		a_0=\mu^3+(2-s)\mu^2+(1-s-st)\mu-st=(\mu+1)(\mu^2-(s-1)\mu-st).
	\end{array}\right.
	$$
	Moreover, by regarding $A(H)$ as $C$ in  Lemma \ref{lem-55},  we obtain
	\begin{equation}\label{eq-6}
		\begin{aligned}
			m(\mu)(\mu I-A(H))^{-1}=&A(H)^3+(\mu+2-s)A(H)^2+(\mu^2+2\mu-s\mu+1-s-st)A(H)\\
			&+(\mu+1)(\mu^2-(s-1)\mu-st)I\\
			=&\begin{pmatrix}
				\alpha J_{s\times s}+\beta\mu I_{s\times s} &\delta J_{s\times t} \\
				\delta J_{t\times s} &\gamma J_{t\times t}+ \beta(\mu+1)I_{t\times t}
			\end{pmatrix},
		\end{aligned}
	\end{equation}
	where $\alpha=\mu^2+\mu t$, $\beta=\mu^2-(s-1)\mu-st$, $\gamma=(\mu s+s)$ and $\delta=(\mu^2+\mu)$.

	 If $u,v$ are some vertices of the star set $X$, then we suppose that $u$ is of type $(c_1,c_2,\ldots,c_s,b)$, where $\sum\limits_{i=1}^{s}c_i=a$, and $v$ is of type $(e_1,e_2,\ldots,e_s,f)$, where $\sum\limits_{i=1}^{s}e_i=e$. Let $N_{K_s}(u)=Y_1$, $N_{tK_1}(u)=Z_1$ and $N_{K_s}(v)=Y_2$, $N_{tK_1}(v)=Z_2$. Recall that $\mathbf{b}_u$ and $\mathbf{b}_v$ are the columns of $B$ corresponding to $u$ and $v$, respectively.  From (\ref{X-eq-1}) we see that $\mathbf{b}_u$ has the form $\mathbf{b}_u=(\mathbf{b}_{Y_{1}}^\intercal ,\mathbf{b}_{Z_{1}}^\intercal )^\intercal $, where $\mathbf{b}_{Y_1}$ and $\mathbf{b}_{Z_1}$ are the characteristic vectors of $Y_1$ and $Z_1$ (i.e.~they determine the $Y_1$-neighbourhood and $Z_1$-neighbourhood of $u$), respectively. Similarly, $\mathbf{b}_v=(\mathbf{b}_{Y_{2}}^\intercal ,\mathbf{b}_{Z_{2}}^\intercal )^\intercal $, where $\mathbf{b}_{Y_2}$ and $\mathbf{b}_{Z_2}$ are the characteristic vectors of $Y_2$ and $Z_2$, respectively. Then
	$$
	\left\{\begin{array}{l}
		|N_{K_s}(u)|=|Y_1|=a, \\
		|N_{tK_1}(u)|=|Z_1|=b, \\
		|N_{K_s}(v)|=|Y_2|=e, \\
		|N_{tK_1}(v)|=|Z_2|=f,
	\end{array}\right.
	$$
	and so
	$$
	\left\{\begin{array}{c}
		|N_H(u)|=|Y_1|+|Z_1|=a+b, \\
		|N_H(v)|=|Y_2|+|Z_2|=e+f.
	\end{array}\right.
	$$
We further denote $\rho_{s}=|N_{K_s}(u)\cap N_{K_s}(v)|$ and $\rho_t=|N_{tK_1}(u)\cap N_{tK_1}(v)|$, along with  $U=(\mu I-A_X)-B^\intercal (\mu I-A(H))^{-1}B$ and  $f(\mu;u,v)=m(\mu)U_{uv}$, where $U_{uv}$ stands for the $(u,v)$-entry of $U$. But, from (\ref{equation-1}) we know that $U$ is an all-0 matrix, which yields
	$$f(\mu;u,v)=m(\mu)U_{uv}=m(\mu)\big((\mu I-A_X)_{uv}-\mathbf{b}_u^\intercal  (\mu I-A(H))^{-1}\mathbf{b}_v\big)=0.$$
	Combining this with (\ref{eq-6}), we obtain
	\begin{equation*}
		\begin{aligned}
			-a_{uv} m(\mu)
			=&\,\mathbf{b}_u^\intercal  m(\mu)(\mu I-A(H))^{-1}\mathbf{b}_v\\
			=&\,(\mathbf{b}_{Y_1}^\intercal ,\mathbf{b}_{Z_1}^\intercal )\begin{pmatrix}
				\alpha J_{s\times s}+\beta\mu I_{s\times s} &\delta J_{s\times t} \\
				\delta J_{t\times s} &\gamma J_{t\times t}+ \beta(\mu+1)I_{t\times t}
			\end{pmatrix}\begin{pmatrix}
				\mathbf{b}_{Y_2} \\
				\mathbf{b}_{Z_2}
			\end{pmatrix}\\
			=&\,\begin{pmatrix}
				\mathbf{b}_{Y_1}^\intercal (\alpha J_{s\times s}+\beta\mu I_{s\times s})+\delta\mathbf{b}_{Z_1}^\intercal J_{t\times s} & \delta\mathbf{b}_{Y_1}^\intercal  J_{s\times t}+\mathbf{b}_{Z_1}^\intercal (\gamma J_{t\times t}+ \beta(\mu+1)I_{t\times t})
			\end{pmatrix}\begin{pmatrix}
				\mathbf{b}_{Y_2} \\
				\mathbf{b}_{Z_2}
			\end{pmatrix}\\
			=&\,\mathbf{b}_{Y_1}^\intercal (\alpha J_{s\times s}+\beta\mu I_{s\times s})\mathbf{b}_{Y_2}+\delta\mathbf{b}_{Z_1}^\intercal J_{t\times s} \mathbf{b}_{Y_2}+ \delta\mathbf{b}_{Y_1}^\intercal  J_{s\times t}\mathbf{b}_{Z_2}+\mathbf{b}_{Z_1}^\intercal (\gamma J_{t\times t}+ \beta(\mu+1)I_{t\times t})\mathbf{b}_{Z_2}\\
			=&\,\alpha ae+\beta\mu \rho_{s}+\delta be+\delta af+\gamma bf+\beta(\mu+1)\rho_{t}.
		\end{aligned}
	\end{equation*}
	Therefore,
	\begin{equation}\label{eq-22}
		\begin{aligned}
			f(\mu;u,v)=&\,-a_{uv} m(\mu)-\big(\alpha ae+\beta\mu \rho_{s}+ (be+ af)\delta+\gamma bf+\beta(\mu+1)\rho_{t}\big)\\
			=&\,(-a_{uv}\mu-\rho_{s}-\rho_{t})(\mu+1)(\mu^2-(s-1)\mu-st)\\
			&\,-(ae-\rho_{s}+be+af)\mu^2-(aet+\rho_{s}(s-1)+be+af+sbf)\mu\\
			&\,-st\rho_{s}-sbf=0.\\
		\end{aligned}
	\end{equation}
	Similarly, for $u=v$ we have
	$$f(\mu;u,u)=m(\mu)U_{uu}=m(\mu)\big((\mu I-A_X)_{uu}-\mathbf{b}_u^\intercal  (\mu I-A(H))^{-1}\mathbf{b}_u\big)=0.$$
	Combining this with (\ref{eq-6}),  we obtain
	\begin{equation*}
		\begin{aligned}
			\mu m(\mu)
			=&\,\mathbf{b}_u^\intercal  m(\mu)(\mu I-A(H))^{-1}\mathbf{b}_u\\
			=&\,\alpha a^2+\beta\mu a+\delta ab+\delta ab+\gamma b^2+\beta(\mu+1)b.
		\end{aligned}
	\end{equation*}
	Hence,
	\begin{equation}\label{eq-1}
		\begin{aligned}
			f(\mu;u,u)=&\,\mu m(\mu)-(\alpha a^2+\beta\mu a+\delta ab+\delta ab+\gamma b^2+\beta(\mu+1)b)\\
			=&\,\mu^5 + (2 - s)\mu^4 + (1 - b - s - st - a)\mu^3\\
			&\,+ (as-2b - 2ab + bs-st-a^2-a)\mu^2\\
			&\,+ (bs - 2ab-b-a^2t - b^2s + ast + bst)\mu\\
			&\,- sb^2 + stb=0.
		\end{aligned}
	\end{equation}

	The equality \eqref{eq-1} can also be in~\cite{Peter2}; here, we reproduce it for the sake of completeness.  Since $\mu\neq r$, from Lemma~\ref{lem-4} we have $\langle \mathbf{b}_u,\mathbf{j} \rangle=-1$. By multiplying both sides of this equality by $m(\mu)$, we get \begin{equation*}
		\begin{aligned}
			-m(\mu)
			=&\,\mathbf{b}_u^\intercal  m(\mu)(\mu I-A(H))^{-1}\mathbf{j}\\
			=&\,(\mathbf{b}_{Y_1}^\intercal ,\mathbf{b}_{Z_1}^\intercal )\begin{pmatrix}
				\alpha J_{s\times s}+\beta\mu I_{s\times s} &\delta J_{s\times t} \\
				\delta J_{t\times s} &\gamma J_{t\times t}+ \beta(\mu+1)I_{t\times t}
			\end{pmatrix}\begin{pmatrix}
				\mathbf{j}_s \\
				\mathbf{j}_t
			\end{pmatrix}\\
			=&\,\alpha sa+\beta\mu a+\delta sb+\delta ta+\gamma tb+\beta(\mu+1)b,
		\end{aligned}
	\end{equation*}
	which gives
	\begin{equation}\label{eq-3}
		a(\mu+t)+b(\mu+1)=s(\mu+t)-\mu(\mu+1),
	\end{equation}
	 equivalently, $b=\frac{(s-a)(\mu+t)}{\mu+1}-\mu$. Combining this with (\ref{eq-1}), we get
	\begin{equation}\label{ee-1}
		\begin{split}
			&\big((\mu^2-(s-1)\mu - st)((t +\mu)a^2 + (t + 2\mu - 2st - 2s\mu + t\mu + 2\mu^2)a\\& + x - st - 2s\mu + s^2t - 2s\mu^2 + s^2\mu + 3\mu^2 + 3\mu^3 + \mu^4 - st\mu)\big)/(\mu+ 1)=0.
		\end{split}
	\end{equation}
	Since $\mu^2-(s-1)\mu - st\not=0$ and $\mu\neq-1$ (since $-1$ is an eigenvalue of the star complement), by taking into account (\ref{ee-1}), we arrive at
	\begin{equation}\label{eq-4}
		\begin{split}
			&(t +\mu)a^2 + (t + 2\mu - 2st - 2s\mu + t\mu + 2\mu^2)a+\mu - st - 2s\mu\\
			&+ s^2t - 2s\mu^2 + s^2\mu + 3\mu^2 + 3\mu^3 + \mu^4 - st\mu=0.
		\end{split}
	\end{equation}
	We record this as the following lemma.
	\begin{lem}\label{lem-66}
		Let $K_s\nabla tK_1$ $(s,t\geq2)$ be a star complement for an eigenvalue $\mu$ in an $r$-regular graph $G$, and let $X$ denotes the corresponding star set. The parameter $|N_{K_s}(u)|=a$ satisfies the equation (\ref{eq-4}).
	\end{lem}

In what follows we first consider the case in which $\mu=-t$.
{It follows that $t + 2\mu - 2st - 2s\mu + t\mu + 2\mu^2=\mu(\mu+1)\not=0$.  Thus, the equality (\ref{eq-4}) is linear in $a$, which leads to
$a=-\mu^2 - 2\mu + s-1$.
On the other hand, from (\ref{eq-3}) we have $b=-\mu=t$. Combining this with Lemma~\ref{lem-66} we immediately get the following corollary.
}

{
\begin{cor}\label{a-b-cor-1}
Under the assumption of Lemma \ref{lem-66}, if $\mu=-t$ then
\begin{equation*}\label{eq-28}
\left\{\begin{array}{ll}
|N_{K_s}(u)|=a=-\mu^2 - 2\mu + s-1,\\
|N_{tK_1}(u)|=b=-\mu=t
\end{array}\right.
\end{equation*}
holds for all $u\in X$.
\end{cor}

We are now in position to prove a conditional resolution of Conjecture~\ref{con-1}.

\begin{thm}\label{lem-31} Conjecture~\ref{con-1} holds for $\mu=-t$.
\end{thm}
\begin{proof} Let $G$ be an $r$-regular graph satisfying the assumptions of the conjecture, along with $\mu=-t$. We need to prove that $G=\overline{(s+1)K_2}$.
	
	From Corollary~\ref{a-b-cor-1} we see that every vertex of $X$ is adjacent to all vertices of $tK_1$. In other words, $|N_X(w)|=|X|$ holds for every $w \in V(tK_1)$. Since $G$ is $r$-regular and $G-X=K_s\nabla tK_1$ is a star complement for $\mu$, we have
\begin{equation*}\label{bb-1}
r=d(w)=s+|X|.
\end{equation*}
Suppose that $|N_X(w')|= c$ for $w' \in V(K_s)$ and $|N_X(u)|=d\leq |X|-1$ for $u\in X$.  As above, we get
\begin{equation*}\label{bb-2}
r=d(w')=s-1+t+c,
\end{equation*}
while from the first equality of~Corollary~\ref{a-b-cor-1}, we obtain
\begin{equation*}\label{bb-3}
r=d(u)=a+b+d=-\mu^2 - 2\mu + s-1-\mu+d.
\end{equation*}
Combining the previous equalities, we get
\begin{equation}\label{bb-4}
|X|-t+1 =c=a+b+d-s+1-t=-\mu^2-2\mu+d,
\end{equation}
which leads to $-\mu^2-2\mu=|X|-t+1-d=|X|+\mu+1-d$. Hence, $-\mu^2-3\mu-1=|X|-d\geq 1$, and thus $\mu\in\{-1, -2\}$. Moreover, since  $\mu\neq -1$, we have $\mu=-2$ (and then $b=t=2$). It follows from Corollary~\ref{a-b-cor-1} that $a=s-1$, while from (\ref{bb-4}) we have $c=d=|X|-1$. Consequently, $G-X=K_s\nabla 2K_1$ and $X$ induces a clique.

 It follows from the previous computation that every $u\in X$ is adjacent to $s-1$ vertices of $K_s$ and both vertices of $2K_1$. On the other hand, by Lemma~\ref{lem-5}, we have $N_H(u)\not=N_H(v)$, for distinct $u, v\in X$. This implies that  $N_H(u)\cap V(K_s)\not=N_H(v)\cap V(K_s)$, and thus $|X|\leq \binom{s}{s-1}=s$. Taking into account that $c=|N_X(w')|=|X|-1$, we get that every $w'\in K_s$ is adjacent to $|X|-1$ neighbours of $X$. By counting the number of  edges between $X$ and $V(K_s)$  we get $(|X|-1)s=|X|(s-1)$, and so $|X|=s$. Therefore, $G$ is obtained from  $K_{2(s+1)}$ by deleting a perfect matching, i.e. $G=\overline{(s+1)K_2}$.
\end{proof}
}
	
 From this point we assume that $\mu\neq -t$. This  implies that the equation (\ref{eq-4}) is quadratic in $a$ with roots
	\begin{equation}\label{eq-29}
		\left\{\begin{aligned}
			a_1=s-\frac{(\mu+1)(2\mu+t+\sqrt{t^2-4\mu^2t-4\mu^3})}{2(\mu+t)}, \\
			a_2=s-\frac{(\mu+1)(2\mu+t-\sqrt{t^2-4\mu^2t-4\mu^3})}{2(\mu+t)}.
		\end{aligned}\right.
	\end{equation}
	 The next result  follows immediately from  (\ref{eq-3}) and (\ref{eq-29}).}
	\begin{cor}\label{a-b-cor-2}
		Let, under the assumptions of Lemma~\ref{lem-66},  $\mu\neq -t$.
			\begin{itemize}
				\item[(i)] If $t^2-4\mu^2t-4\mu^3=0$, then $|N_{K_s}(u)|=a=s-\frac{(\mu+1)(2\mu+t)}{2(\mu+t)}$ and $
				|N_{tK_1}(u)|=b=\frac{t}{2}$, where $a, b$ are defined in the beginning of this section.
			
				\item[(ii)] If $t^2-4\mu^2t-4\mu^3\not=0$,  then $|N_{K_s}(u)|=a_1$ or $a_2$, and $|N_{tK_1}(u)|=b_1$ or $b_2$, where the latter parameters are obtained by inserting $a_1, a_2$ of \eqref{eq-29} into \eqref{eq-3}.
			\end{itemize}
	\end{cor}

	The two cases that arise from the previous corollary are considered in the forthcoming Theorems~\ref{lem-32} and~\ref{lem-33}.

	\begin{thm}\label{lem-32}
		{If {$\mu\neq -t$} and $t^2-4\mu^2t-4\mu^3=0$, then there is no regular graph $G$ with the star complement $K_s\nabla tK_1$ $(s,t\geq2)$ for an eigenvalue $\mu$.}
	\end{thm}
	\begin{proof}
		Under the given assumptions the equation (\ref{eq-4}) has a single root
		\begin{equation}\label{eq-25}
			a_1=a_2=a=s-\frac{(\mu+1)(2\mu+t)}{2(\mu+t)}.
		\end{equation}
		
		Combining this with (\ref{eq-3}) we obtain that
		$b=\frac{t}{2}$ must be an integer.
		Observe now that $\mu$ is a root of  $f(x)=t^2-4x^2t-4x^3$ due to
		\begin{equation}\label{eq-26}
			t^2-4\mu^2t-4\mu^3=0.
		\end{equation}
		
		Suppose first  that  $f$ has a rational root $\mu$, which in fact must be an integer. From (\ref{eq-26}) we get $t=2\mu^2\pm\sqrt{4\mu^4+4\mu^3}=2\mu^2\pm2|\mu|\sqrt{\mu(\mu+1)}$ and  $\mu(\mu+1)\geq0$ (i.e. $\mu<-1$ or $\mu>0$ due to  $\mu\notin\{0, -1\}$). Since  $(\mu+1)^2>\mu(\mu+1)>\mu^2$ if $\mu>0$, and $(\mu+1)^2<\mu(\mu+1)<\mu^2$ if $\mu<-1$, we see that $\sqrt{\mu(\mu+1)}$ is not an rational number, which implies that $t$ is not an integer, a contradiction.
		
		Suppose now that $f$ has an irrational root $\mu$. We know that $t+\mu=\frac{t^2}{4\mu^2}$ due to (\ref{eq-26}). Hence,
			\begin{align}
				a=&\,s-\frac{(\mu+1)(2\mu+t)}{2(\mu+t)}\nonumber\\ \label{eq-9}
				=&\,s-\mu-1+\frac{t}{2}-\frac{t^2-t}{2(\mu+t)}\\ \label{eq-10}
				=&\,s-\mu-1+\frac{t}{2}-\frac{2\mu^2(t-1)}{t},
		\end{align}
		which gives  $\frac{2\mu^2(t-1)}{t}+\mu=s-1+\frac{t}{2}-a$. Hence,  $\frac{2\mu^2(t-1)}{t}+\mu$ is an integer, say $z$, where  $z=s-1+\frac{t}{2}-a$. Thus $2(t-1)\mu^2+t\mu-tz=0$ is a quadratic equation with integral coefficients, and therefore $\mu=l+g\sqrt{h}$, where $l=\frac{-t}{4(t-1)}, g=\pm\frac{1}{4(t-1)}, h=t^2+8(t-1)tz\in \mathbb{Q}$,  $g,h\not=0$ and $h$ is not an square because $\mu$ is irrational. Replacing for $\mu$ in (\ref{eq-9}), we get
		$$\begin{aligned}
			a=&\,s-\mu-1+\frac{t}{2}-\frac{t^2-t}{2(\mu+t)}\\
			=&\,s-l-g\sqrt{h}-1+\frac{t}{2}-\frac{(t^2-t)}{2(t+l+g\sqrt{h})}\\
			=&\,s-l-1+\frac{t}{2}-\frac{(t^2-t)(t+l)}{2\big((t+l)^2-g^2h\big)}+\bigg(\frac{(t^2-t)}{2\big((t+l)^2-g^2h\big)}-1\bigg)g\sqrt{h}.
		\end{aligned}
		$$
		It follows that the last term $\bigg(\frac{(t^2-t)}{2\big((t+l)^2-g^2h\big)}-1\bigg)g\sqrt{h}$ must be zero, and so
		\begin{equation}\label{eq-11}
			t^2-t=2\big((t+l)^2-g^2h\big).
		\end{equation}
		
		Similarly, replacing for $\mu$ in (\ref{eq-10}), we obtain
		$$\begin{aligned}
			a=&\,s-\mu-1+\frac{t}{2}-\frac{2\mu^2(t-1)}{t}\\
			=&\,s-l-g\sqrt{h}-1+\frac{t}{2}-\frac{2(t-1)}{t}(l^2+g^2h+2lg\sqrt{h})\\
			=&\,s-l-1+\frac{t}{2}-\frac{2(t-1)}{t}(l^2+g^2h)-g\bigg(1+\frac{4l(t-1)}{t}\bigg)\sqrt{h}.
		\end{aligned}$$
		It follows that $1+\frac{4l(t-1)}{t}=0$, and so
		\begin{equation}\label{eq-12}
			l=\frac{-t}{4(t-1)}.
		\end{equation}
		Moreover, from (\ref{eq-26}) we obtain $$\begin{aligned}
			0=&\,t^2-4\mu^2t-4\mu^3\\
			=&\,t^2-4\mu^2(t+\mu)\\
			=&\,t^2-4\big(l^3+l^2t+tg^2h+3lg^2h+g(3l^2+g^2h+2lt)\sqrt{h}\big),
		\end{aligned}
		$$
		which yields
		\begin{equation*}\label{eq-13}
			3l^2+g^2h+2lt=0.
		\end{equation*}
		Combining this with (\ref{eq-11}) and (\ref{eq-12}), we get
		\begin{equation*}\label{ttt-1}
			2t^4-6t^3+3t^2+2t=0.
		\end{equation*}
		The latter equation has roots: $0$, $2$, $\frac{1}{2}(1+\sqrt{3})$ and $\frac{1}{2}(1-\sqrt{3})$. Therefore, $t=2$ since it is an integer. Replacing for $t$ in  (\ref{eq-12}) and (\ref{eq-13}), we obtain $l=-\frac{1}{2}$ and $g^2h=\frac{5}{4}$, and so $\mu=l+g\sqrt{h}=-\frac{1}{2}\pm \frac{\sqrt{5}}{2}$. However, from  (\ref{eq-25}) we have $a=s-1$, which implies $\mu=-2$ by Lemma~\ref{lem-3-1}, a contradiction.
	\end{proof}

 	\begin{figure}
	\centering
	\includegraphics[width=90mm,angle=0]{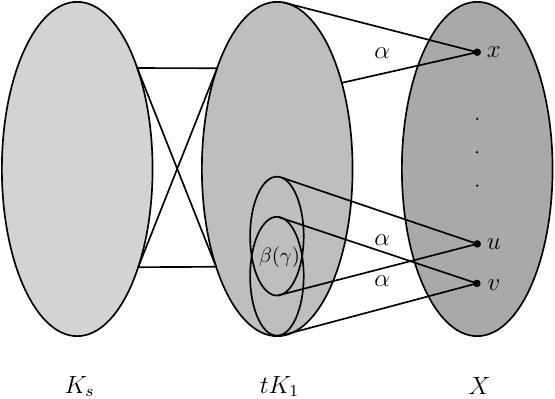}
	\caption{A sketch of an  $Y$-graph $G(r; \alpha, \beta, \gamma)$.}\label{fig-1}
\end{figure}

Evidently, the result of the previous theorem confirms Conjecture~\ref{con-1} under the assumption that $t^2-4\mu^2t-4\mu^3=0$.

Let further $G=G(r; \alpha, \beta, \gamma)$ be a putative $r$-regular graph with vertex set partition $V(G)=U\cup X$ satisfying the following conditions:
	
\begin{itemize}
	\item $G[U]=K_s\bigtriangledown tK_1$;
	\item For  $x\in X$, $|N_{G}(x)\cap V(K_s)|=0$ and $|N_{G}(x)\cap V(tK_{1})|=\alpha$;
	\item For each pair of vertices $u, v\in X$,
	$$
	|N_{tK_1}(u)\cap N_{tK_1}(v)|=\left\{\begin{array}{ll}
		\beta & \mbox{ if~\, $u\sim v$},\\
		\gamma  & \mbox{ if~\, $u\not\sim v$}.
	\end{array}\right.
	$$	
\end{itemize}
	
	We call  $G$ the \textit{$Y$-graph} with parameters $r,\alpha,\beta$ and $\gamma$. It is sketched in Fig.~\ref{fig-1}. We remark  that for each  $x\in X$ we have $|N_{X}(x)|=r-\alpha$, for each $w\in V(tK_1)$ we have $|N_{X}(w)|=r-s$ and the order of a $Y$-graph is $|X|+s+t$.

	\begin{thm}\label{lem-33}
		Suppose that an $r$-regular graph $G$ of order $n$  contains $K_s\nabla tK_1$ $(s,t\geq2)$ as a star complement for an eigenvalue $\mu$, with the corresponding star set $X$. If {$\mu\neq -t$} and $t^2-4\mu^2-4\mu^3\not=0$, then $G$ is a $Y$-graph $G(r;\alpha,\beta,\gamma)$ such that
			\begin{equation}\label{1set}
			\left\{\begin{aligned}
				r=&\,\frac{\mu^2(\mu+1)^2+(\mu+1-s)\big(s^2+(\mu+1)(\mu-s)\big)}{s(\mu+1-s)},\\
				\alpha=&\,\mu^2+\frac{s\mu^2}{\mu+1-s},\\
				\beta=&\,\frac{\mu(\mu+1)(s-1)}{\mu+1-s},\\
				\gamma=&\,\frac{s\mu^2}{\mu+1-s}.
			\end{aligned}\right.
			\end{equation}
			Moreover,
			\begin{equation}\label{2set}
			\left\{\begin{aligned}
				|N_{X}(u)|=&\,s-(\mu+1)+\frac{\mu(\mu+1)^2}{s},\ \ \mbox{for $u\in X$,}\\
				|N_{X}(w)|=&\,\frac{\mu^2(\mu+1)^2+(\mu+1-s)(\mu+1)(\mu-s)}{s(\mu+1-s)}, \ \ \mbox{for $w\in V(tK_1)$,}\\
				t=&\,\frac{\mu\big(\mu(\mu+1)^2+(\mu+1-s)^2\big)}{s(\mu+1-s)},\\
				|X|=&\,\frac{\big(\mu(\mu+1)^2+(\mu+1-s)^2\big)\big(s+(\mu-1)(\mu+1)^2+(\mu+1-s)^2\big)}{\mu s^2(\mu+1-s)},\\
				n=&\,\frac{\big((\mu+1)^2+s(\mu-1)\big)\big(\mu^4+3\mu^3-3(s-1)\mu^2+(3s^2-4s+1)\mu-s^3+2s^2-s\big)}{\mu s^2(\mu+1-s)}.
			\end{aligned}\right.
			\end{equation}
	\end{thm}
	\begin{proof} This is a considerably long proof and we begin with a short concept. We divide the proof into the 3 parts.
		
		\begin{itemize}
			\item In Part 1 we perform some initial computation; in particular,  we express $a_1, a_2, b_1, b_2$ and~$t$, and show that $\mu$  is an integer satisfying $\mu+t>0$.

			\item In Part 2 we eliminate the possibility that $\mu$ is negative.
			
			\item In Part 3 we deal with $\mu$ being a positive integer. In an intermediate step we prove that for distinct $u, v\in X$, there must be $a=|N_{K_s}(u)\cap N_{K_s}(v)|=0$. Then we show that a putative graph satisfying the assumptions of the statement must be a $Y$-graph and compute the parameters of \eqref{1set} and \eqref{2set}.
		\end{itemize}

{\flushleft\bf Part 1.}	By solving the equation (\ref{eq-4}) under the assumptions given in the formulation of this statement, we get the roots
		\begin{equation}\label{eq-7}
			\left\{\begin{aligned}
				a_1=&\,s-\frac{(\mu+1)(2\mu+t+\sqrt{t^2-4\mu^2t-4\mu^3})}{2(\mu+t)}, \\
				a_2=&\,s-\frac{(\mu+1)(2\mu+t-\sqrt{t^2-4\mu^2t-4\mu^3})}{2(\mu+t)}.
			\end{aligned}\right.
		\end{equation}
		On the other hand, (\ref{eq-3}) gives $b=\frac{(s-a)(\mu+t)}{(\mu+1)}-\mu$. Replacing $a$ with $a_1$ and then with $a_2$, we get the following two possibilities for $b$:
		\begin{equation*}\label{eq-8}
			\left\{\begin{aligned}
				b_1=&\,\frac{t+\sqrt{t^2-4\mu^2t-4\mu^3}}{2},\\
				b_2=&\,\frac{t-\sqrt{t^2-4\mu^2t-4\mu^3}}{2}.
			\end{aligned}\right.
		\end{equation*}
		Obviously, $\sqrt{t^2-4\mu^2t-4\mu^3}$ is a non-negative integer because $b_1, b_2$ are integers, and so we may set
		\begin{equation}\label{tt-0}
		\sqrt{t^2-4\mu^2t-4\mu^3}=p\in \mathbb{N}.\end{equation}

		From  (\ref{eq-7}) we obtain  $a_2-a_1=\frac{(\mu+1)\sqrt{t^2-4\mu^2t-4\mu^3}}{\mu+t}=\frac{p(\mu+1)}{\mu+t}$. We set $\frac{p(\mu+1)}{\mu+t}=q$. Evidently, $q$ is rational since $a_1,a_2$ are. Moreover, $\mu$ is also rational, since for otherwise from $p(\mu+1)=q(\mu+t)$, we have $p-qt=(q-p)\mu$, and thus $p=q$, $\mu+1=\mu+t$, which leads to the impossible scenario $t=1$. We further have $\mu\in \mathbb{Z}$ since it is an algebraic integer.
		
		 Next, from (\ref{eq-3}) we have $a=s-\frac{(\mu+1)(\mu+b)}{\mu+t}$, which means that $\frac{(\mu+1)(\mu+b)}{\mu+t}\geq0$ because $a\leq s$. Consequently, it holds
		\begin{equation}\label{eq-35}
			\mu+t>0,
		\end{equation}
		as for otherwise we would have $\mu+b\leq \mu+t<0$, $\mu+1< \mu+t<0$, and then $\frac{(\mu+1)(\mu+b)}{\mu+t}<0$, which contradicts the previous conclusion.

			{\flushleft\bf Part 2.} Here we eliminate the possibility that $\mu$ is negative. By way of contradiction we have $\mu\leq -2$ (as $\mu$ is an integer distinct from $-1$). We first notice that $2\mu+t>0$. Namely, if $2\mu+t\leq 0$ then $t\leq-2\mu$, and so $t^2\leq4\mu^2$. By  (\ref{eq-35}), we get that $\mu+t\geq1$ is an integer, and then from  (\ref{tt-0}) we obtain  $$0<p^2=t^2-4\mu^2t-4\mu^3= t^2-4\mu^2(t+\mu)\leq4\mu^2-4\mu^2(t+\mu)\leq0,$$
		a contradiction.

		From $2\mu+t>0$ and the first equality of (\ref{eq-7}), we see that $$s-a_1=\frac{(\mu+1)(2\mu+t+\sqrt{t^2-4\mu^2t-4\mu^3})}{2(\mu+t)}<0,$$
		since $2\mu+t+\sqrt{t^2-4\mu^2t-4\mu^3}>0$ and $2(\mu+t)>0$, a contradiction.
		
		Similarly, by taking into account the second equality of (\ref{eq-7}), we get $$s-a_2=\frac{(\mu+1)(2\mu+t-\sqrt{t^2-4\mu^2t-4\mu^3})}{2(\mu+t)}.$$
		On the other hand, we  also have $2\mu+t-\sqrt{t^2-4\mu^2t-4\mu^3}>0$, since for  otherwise we would have $0<2\mu+t\leq\sqrt{t^2-4\mu^2t-4\mu^3}$, which gives  $\mu\in \{0, -1\}$. Thus, $s-a_2<0$, which is impossible.

		{\flushleft\bf Part 3.} Here we assume that $\mu$ is a positive integer. It follows that $q=\frac{p(\mu+1)}{\mu+t}>0$.
		Since $t>0$, we have $t=2\mu^2+\sqrt{4\mu^4+4\mu^3+p^2}$ from (\ref{tt-0}), and therefore
		$$
			0=p(\mu+1)-q(\mu+t)=p(\mu+1)-q(\mu+2\mu^2+\sqrt{4\mu^4+4\mu^3+p^2}),
		$$
		which gives  $p(\mu+1)-q(\mu+2\mu^2)=q\sqrt{4\mu^4+4\mu^3+p^2}>0$. We claim that at least one of $a_1, a_2$ is not an integer. Indeed, by assuming that
		$a_1,a_2\in \mathbb{Z}$, we immediately get $q\in \mathbb{Z}$. Further, from (\ref{eq-7}) we obtain $a_1=s-\frac{(\mu+1)(2\mu+t+p)}{2(t+\mu)}$, which  gives  $\frac{q\mu}{p}=2(s-a_1)-\mu-1-q$. Now since $\frac{q\mu}{p}$ is a positive integer (due to $q,p,\mu>0$), we have  $p\leq q\mu$ and $p(\mu+1)-q(\mu +2\mu^2)\leq q\mu(\mu+1)-q(\mu+2\mu^2)=-q\mu^2<0$, a contradiction.
		
		Therefore, exactly one of $a_1, a_2$ is an integer, which leads to the following settings: If $a_1$ is an integer, then $a=a_1$ and $b=b_1$, otherwise $a=a_2$ and $b=b_2$. We also have that all vertices in $X$ have $a$ neighbours in $K_s$ and $b$ neighbours in $tK_1$ (since  $u\in X$ is chosen arbitrarily).
		
		Next, taking into account regularity of $G$ and conditions $s,t\geq 2$, we easily conclude that $|X|\geq 2$. Let $v\in X$, $v\not=u$, and suppose that $v$ is of type $(e_1,e_2,\ldots,e_s,f)$. Then $\sum\limits_{i=1}^{s}e_i=a$ and $f=b$. Let $\rho_{s}=|N_{K_s}(u)\cap N_{K_s}(v)|$ and $\rho_{t}=|N_{tK_1}(u)\cap N_{tK_1}(v)|$.  From (\ref{eq-22}) we obtain
						$(-a_{uv}\mu-\rho_{s}-\rho_{t})(\mu+1)\big(\mu^2-(s-1)\mu-st\big)
				=(a^2-\rho_{s}+2ab)\mu^2+\big(a^2t+\rho_{s}(s-1)+2ab+sb^2\big)\mu+st\rho_{s}+sb^2$.
			This equality leads to
		\begin{equation}\label{eq-23}
			\rho_t=-a_{uv}\mu-\frac{\rho_s\mu}{\mu+1}-\frac{(a^2+2ab)\mu^2+(a^2t+2ab+sb^2)\mu+sb^2}{(\mu+1)(\mu^2-(s-1)\mu-st)}.
		\end{equation}
		 We shall return to the previous equality soon. At this point, by expressing $t$ from (\ref{eq-3}) we get
		 \begin{equation}\label{eq-18}
		 	t=-\mu+\frac{(b+\mu)(\mu+1)}{s-a},
		 \end{equation}
	where $s-a>0$ (since  $s=a$ implies $b=-\mu<0$). Replacing for $t$ in (\ref{eq-1}), we get
		\begin{equation*}\label{eq-30}
			\frac{(\mu+1)(bs+a\mu)(-\mu^3-\mu^2+b\mu+b+ab-bs)}{s-a}=0,
		\end{equation*}

		Recall that $\mu\neq -1$, while by Lemma~\ref{lem-5} we know that $a=0$ and $b=0$ cannot simultaneously hold. Thus,  $$0=-\mu^3-\mu^2+b\mu+b+ab-bs=-\mu^2(\mu+1)+b(\mu+1+a-s),$$
		which  can also be written as
		\begin{equation}\label{eq-17}
			b=\frac{\mu^2(\mu+1)}{\mu+1+a-s}=\mu^2+\frac{\mu^2(s-a)}{\mu+1+a-s}
		\end{equation}
		because  $-\mu^2(\mu+1)\not=0$ and then $\mu+1+a-s\not=0$.
		Now by combining (\ref{eq-18}) and (\ref{eq-17}), we obtain
		\begin{equation}\label{eq-21}
				t=\frac{-\mu\big(\mu(\mu+1)^2+(\mu+1+a-s)^2\big)}{(a-s)(\mu+1+a-s)}.
		\end{equation}
		If $u\sim v$,  then we have $a_{uv}=1$ in (\ref{eq-23}) and, together with (\ref{eq-17}) and (\ref{eq-21}), this leads to
		\begin{equation}\label{eq-24}
			\rho_t=\frac{\mu\big((a-\rho_s-\mu-1)(\mu+1+a-s)+\mu(\mu+1)(s-a)\big)}{(\mu+1)(\mu+1+a-s)}.
		\end{equation}
		In a similar way, for $u\not\sim v$ we get
		\begin{equation}\label{eq:new}
			\rho_t=\frac{\mu\big(a-\rho_s)(\mu+1+a-s)+\mu(\mu+1)(s-a)\big)}{(\mu+1)(\mu+1+a-s)}.
		\end{equation}
		
		We now claim that $a=\rho_s=0$ is the unique possibility. In what follows we first show that $\rho_t$ cannot be an integer unless $a=\rho_s$. Assume that $\rho_t$ is an integer when $u\sim v$. From (\ref{eq-24}) we get
		\begin{equation}\label{eq-32}
			\mu\big((a-\rho_s-\mu-1)(\mu+1+a-s)+\mu(\mu+1)(s-a)\big)\equiv 0~ \big(\md~  (\mu+1)(\mu+1+a-s)\big).
		\end{equation}
		By virtue of (\ref{eq-17}), we have that $\frac{\mu^2(s-a)}{\mu+1+a-s}=b-\mu^2$ is an integer due to $b,\mu^2\in \mathbb{Z}$. Then  it holds
		\begin{equation*}\label{eq-31}
			(\mu+1)\mu^2(s-a)\equiv 0~\big(\md~(\mu+1)(\mu+1+a-s)\big).
		\end{equation*}
		Also, it is clear that
		\begin{equation}\label{eq-33}
			\mu(-\mu-1)(\mu+1+a-s)\equiv 0 ~\big(\md~ (\mu+1)(\mu+1+a-s)\big).
		\end{equation}
		Combining  (\ref{eq-32})--(\ref{eq-33}), we deduce that $\mu(a-\rho_s)(\mu+1+a-s)\equiv 0~\big(\md~ (\mu+1)(\mu+1+a-s)\big)$, i.e. $\mu(a-\rho_s)\equiv 0~\big(\md~(\mu+1)\big)$. Since $\mu$ and $\mu+1$ are coprime, we obtain that
		$$(a-\rho_s)\equiv 0~\big(\md~ \mu+1\big),$$
		which is equivalent to
		\begin{equation}\label{eq-34}
			(s-a)-(s-2a+\rho_s)\equiv 0~\big(\md~ \mu+1\big).
		\end{equation}
		We have $0\leq |V(K_s)\backslash (N_{K_s}(u)\cup N_{K_s}(v))|=s-2a+\rho_s\leq s-a$ due to $\rho_s\leq a$. Moreover, it follows from  (\ref{eq-21}) that $0<\mu+1+a-s$, i.e.~$s-a<\mu+1$ since $t$ is a positive integer, and this leads to the conclusion that (\ref{eq-34}) is possible only if $s-a=s-2a+\rho_s$, i.e.~$a=\rho_s$. The other possibility for $\rho_t$ (that is $u\not\sim v$) is considered in a very similar way.
		
		Therefore, all vertices in $X$ share the same neighbourhood in $K_s$. On the contrary, all vertices in $K_s$ have the same number of neighbours in $X$ since $G$ is regular, which implies that $a=0$ or $a=s$. Recalling from the previous part of this proof that $a\not=s$, we get $\rho_s =a=0$, as desired.

		Now, by taking an arbitrary vertex $w'\in V(K_s)$, we see  that
		$$r=d(w')=s-1+t=\frac{\mu^2(\mu+1)^2+(\mu+1-s)\big(s^2+(\mu+1)(\mu-s)\big)}{s(\mu+1-s)},$$
		which is the the first parameter of \eqref{1set}. The remaining three follow by setting $a=0$ in \eqref{eq-17}, \eqref{eq-24} and \eqref{eq:new}, respectively. The equalities of \eqref{1set} assure that $G$ is a $Y$-graph with desired parameters.
		
		The first two  parameters of \eqref{2set} are computed by replacing for $r$ in $|N_{X}(u)|=r-a-b=r-b$ and $|N_{X}(w)|=r-s$, the third follows by setting $a=0$ in \eqref{eq-21}, and the remaining two are computed from $|X|=\frac{(r-s)t}{b}=\frac{t(t-1)}{b}$ and $n=|X|+s+t$.
		
		The proof is complete. \end{proof}
		
		
%
%

	\begin{table}[t]
		{\small\footnotesize
			\begin{tabular}{ccccccccc}
				\toprule
				$\mu$   & $s$ &  $t$ &  $|X|$ & $r$ & $\alpha$ & $\beta$ & $\gamma$ & $n$\\
				\midrule
				9 & 4 &  351 &  910 &354 & 135 & 45 & 54 & 1265\\
				32 & 9 &  5248 &  19557 &5256 & 1408 &352 & 384 & 24814\\
				54 & 10 &  19845 &  110495 & 19854 & 3564 & 594 & 648 & 130350\\
				64 & 25 &  17408 &  45526 & 17432 & 6656 & 2496 & 2560 & 62959\\
				75 & 16 &  34125 &  163436 & 34140 & 7125 & 1425 & 1500 & 197577 \\
				98 & 22 &  55909 &  253139 & 55930 & 12348 & 2646 & 2744 & 309070 \\
				144 & 25 &  146016 &  850915 & 146040 & 25056 & 4176 & 4320 & 996956\\
				245 & 36 &  481915 &  3302874 & 481950 & 70315 & 10045 & 10290 & 3784825\\
				259 & 112 &  273911 &  636658 & 274022 & 117845 & 50505 & 50764 & 910681\\
				384 & 49 &  1330176 &  10472105 & 1330224 & 168960 & 21120 & 21504 & 11802330\\
				441 & 169 &  824229 &  2157538 & 824397 & 314874 & 119952 & 120393 & 2981936 \\
				444 & 75 &  1408960 &  8372840 & 1409034 & 237096 & 39516 & 39960 & 9781875\\
				450 & 121 &  1032750 &  3853917 & 1032870 & 276750 & 73800 & 74250 & 4886788 \\
				567 & 64 &  3219993 &  28617112 & 3220056 & 362313 & 40257 & 40824 & 31837169 \\
				588 & 57 &  3960964 &  40986739 & 3961020 & 382788 & 36456 & 37044 & 44947760\\
				800 & 81 &  7048000 &  69767271 & 7048080 & 712000 & 71200 & 72000 & 76815352\\
				\bottomrule
		\end{tabular}}
		\caption{Feasible parameters for a $Y$-graph with $\mu\leq 800$.}\label{tab-1}
	\end{table}
	
	{From the proof of Theorem~\ref{lem-33} we know that $\mu$ must be a positive integer and the parameters related to $Y$-graph $G$ are uniquely determined by $\mu$ and $s$. In Table~\ref{tab-1} we list the sets of feasible parameters obtained for $\mu\leq 800$. Every row contains $\mu$ (the eigenvalue in question), $s, t$ (the parameters related to the star complement), $|X|$ (the size of the corresponding star set), $r, \alpha, \beta, \gamma$ (the parameters of a putative $Y$-graph $G$) and $n$ (the order of $G$).
		
		However, we were not able to construct any $Y$-graph due to the fact that the corresponding parameters are comparatively large, and the smallest possible example has 1265 vertices. Clearly, the existence of a $Y$-graph would disprove Conjecture~\ref{con-1}.
		
	In what follows we eliminate the two particular cases in which $s\in\{2,3\}$.  The following corollaries can be deduced from the results of \cite{Wang,Yuan}. Here we give the short proofs that rely on the results of this paper.

	\begin{cor}[cf.~\cite{Yuan}]\label{Yuan-1}
		If an $r$-regular graph $G$ has the star complement $K_2\nabla tK_1$ for the eigenvalue $\mu\neq r$, {then $\mu=-t =2$ and  $G=\overline{3K_2}$.}
	\end{cor}
	\begin{proof}
		{If $\mu=-t$, the result follows from Theorem~\ref{lem-31}.} If for $\mu\neq -t$, we also have $t^2-4\mu^2t-4\mu^3=0$, then Lemma~\ref{lem-32} tells us that there is no graph satisfying the assumptions of this corollary. If $t^2-4\mu^2t-4\mu^3\not=0$, then $G$ is a $Y$-graph with $\alpha=\mu^2+\frac{2\mu^2}{\mu-1}$, which yields $\frac{2\mu^2}{\mu-1}=\alpha-\mu^2\in \mathbb{Z}$. Note that $\mu$ and $\mu-1$ are coprime, which implies that $2$ is divisible by $\mu-1$, necessarily $\mu\in\{2, 3\}$. However, we get $|X|=57/2$ for $\mu=2$ and $|X|=247/3$ for $\mu=3$,
		a contradiction.
	\end{proof}

	\begin{cor}[\cite{Wang}]\label{Wang-1}
		If an $r$-regular graph $G$  has the star complement $K_3\nabla tK_1$ for the eigenvalue $\mu\neq r$, {then $\mu=-t=2$ and $G=\overline{4K_2}$}.
	\end{cor}

\begin{proof} As before, the case $\mu=-t$ is settled by Theorem~\ref{lem-31}. For otherwise, $G$ is a $Y$-graph such that
	$$\left\{\begin{aligned}
		\alpha=&\,\mu^2+\frac{3\mu^2}{\mu-2}=\mu^2+3\mu+6+\frac{12}{\mu-2} ,\\
		|X|=&\,\frac{\big(\mu(\mu+1)^2+(\mu-2)^2\big)\big(3+(\mu-1)(\mu+1)^2+(\mu-2)^2\big)}{9\mu(\mu-2)}.
	\end{aligned}\right.$$
	Since $0\leq\alpha\in \mathbb{Z}$ we have $\frac{12}{\mu-2}\in \mathbb{Z}$, i.e. $\mu\in\{1,3,4,5,6,8\}$.
	However, $|X|$ is not a positive integer for any possible $\mu$, and we are done.
\end{proof}

By virtue of Corollaries~\ref{Yuan-1} and~\ref{Wang-1}, we conclude that Conjecture~\ref{con-1} holds for {$\mu\neq -t$}, $t^2-4\mu^2-4\mu^3\not=0$ and $s\in\{2,3\}$.

\section{Relation with block designs}\label{sec:rel}

By the foregoing results, if there is a graph for which Conjecture~\ref{con-1} does not hold, then this graph must be a $Y$-graph defined upon Theorem~\ref{lem-33}. Here we show that its existence depends on the existence of a 2-class block design formed as below.

		Let $G$ be a $Y$-graph and set  $T=X\cup V(tK_1)$.  To construct $G$ it is sufficient to construct its subgraph $G_T$ induced by $T$. {In fact, the existence of $G_T$ depends on the existence of a block design $\mathcal{D}=(X, \mathcal{B})$ whose points are identified with the vertices of $X$, while blocks are determined by the vertices of $tK_1$ in such a way that a point of $X$ belongs to a block of $\mathcal{B}$ if and only if the corresponding vertices are adjacent. If so, then $\mathcal{D}$ has the following parameters. The number of points $|X|$, the block size $k=|N(w)|$ and the number of blocks $t$ are given in \eqref{2set}. The replication (i.e.~the number of occurrences of every point) is $\alpha$, two points joined by an edge occur together in $\beta$ blocks and two non-adjacent points occur together in $\gamma$ blocks, where these parameters are given in \eqref{1set}.  (According to the terminology for block designs, since two points are allowed to occur together in $\alpha$ or $\beta$  blocks, the corresponding design is said to be a \textit{$2$-class block design}.)

		 	Observe that the subgraph $G[X]$ induced by $X$ is regular with vertex degree $r-\alpha=s-(\mu+1)+\frac{\mu(\mu+1)^2}{s}$. Moreover, if $N$ is the incidence matrix whose rows and columns are indexed by $X$ and  $\mathcal{B}$, then
		 	\begin{equation}\label{e:I}
		 		NN^\intercal=(\alpha-\gamma) I +(\beta-\gamma) A(G[X])+\gamma J.
		 	\end{equation}
	 	It is not difficult to see that $G[X]$ is a strongly regular graph if and only if $\mathcal{D}$ is the so-called symmetric 2-class partial incomplete block design \cite[Subsection~3.8.2]{Zoran1}. In this case,  the identity \eqref{e:I} leads to the conclusion that $NN^\intercal$ has exactly 3 eigenvalues.
	 	Moreover, by the same reference, in this particulars case we have an additional condition:
	 	\begin{equation}\label{ac}(r-\alpha)(\beta-\gamma)=\alpha (k-1)-\gamma(|X|-1).\end{equation}

	\begin{example}\label{exm-1} Suppose that $G$ is a $Y$-graph that corresponds to the first row of Table~\ref{tab-1}. Then $G[X]$ is a 219-regular graph (since $|N_{X}(u)|=r-|N_{tK_1}(u)|-|N_{K_s}(u)|=219$) and the existence of $\mathcal{D}$ is conditioned by the existence of such a graph satisfying $NN^{T}=81I-9A(G[X])+54J$. By \eqref{ac} we conclude that $G[X]$ cannot be strongly regular. The search for other possibilities in case of this or any other set of feasible parameters at this moment remains open.
	\end{example}

\section*{Acknowledgements}
This work is supported by the National Natural Science Foundation of China (Grants 11971274, 11531011 and 11671344) and the Serbian Ministry of Education, Science and Technological Development via the Faculty of Mathematics, University of Belgrade.

We are grateful to the referees for their many helpful comments and suggestions, which have improved the presentation of the paper.

\section*{Data availibility} 
Data sharingnot applicable to this article as no datasets were generated or analyzed during the current study.

\end{document}